\documentclass{amsart}
\usepackage{amssymb,mathtools}
\usepackage{tikz-cd}

\usepackage[hidelinks]{hyperref}

\newcommand{\jj}{\mathcal{J}}
\newcommand{\ojj}{\overline{\mathcal{J}}}
\newcommand{\II}{\mathcal{I}}
\newcommand{\Aut}{\mathrm{Aut}}
\newcommand{\Ext}{\mathrm{Ext}}
\newcommand{\C}{\mathbb{C}}
\newcommand{\Q}{\mathbb{Q}}
\newcommand{\et}{\mathrm{H}_{\text{\'et}}}
\newcommand{\cris}{\mathrm{H}_{\mathrm{cris}}}
\newcommand{\OO}{\mathcal{O}}
\newtheorem{theorem}{Theorem}[section]
\newtheorem{cor}[theorem]{Corollary}
\newtheorem{lemma}[theorem]{Lemma}
\newtheorem{prop}[theorem]{Proposition}
\DeclareMathOperator{\Ker}{Ker}
\DeclareMathOperator{\Id}{Id}
\DeclareMathOperator{\Spec}{Spec}
\DeclareMathOperator{\Ho}{H}
\newcommand{\F}{\mathrm{F}}
\DeclareMathOperator{\DR}{DR}
\theoremstyle{definition}
\newtheorem{defi}[theorem]{Definition}
\newtheorem{Assumption}[subsection]{Notation and Assumptions}
\newtheorem{remark}[theorem]{Remark}
\newtheorem{conjecture}[theorem]{Conjecture}
\numberwithin{equation}{section}

\title{Deformation of Automorphisms of Hodge Type}
\author{Renjie Lyu}
\address[Renjie Lyu]{School of Mathematical Sciences, Xiamen University, Xiamen 361005, China}
\email{r.lyu@xmu.edu.cn}
\author{Xuanyu Pan}
\address[Xuanyu Pan]{School of Mathematical Sciences, University of Chinese Academy of Sciences, Beijing 100049, China}
\email{panxuanyu@ucas.ac.cn}
\date{\today}
\begin{document}

\begin{abstract} 
Using de Rham cohomology, we study deformation of automorphisms on a 
central fiber of a family of smooth projective varieties. Motivated
by the variational Hodge conjecture, we show that, under an infinitesimal
Torelli assumption, an automorphism extends to a nearby family if the corresponding
de Rham cohomology endomorphism preserves the Hodge filtration. In mixed
characteristic, using the crystalline and de Rham cohomology comparison, 
we prove an analogous lifting criterion, which generalizes a result of Berthelot and Ogus.
\end{abstract}
\maketitle
\tableofcontents

\section {Introduction}
Let $\pi:X\rightarrow S$ be a smooth projective morphism over a smooth 
connected quasi-projective variety $S$ over \(\mathbb C\). Suppose that
\(\alpha\in \Gamma(S^{\mathrm{an}}, \mathrm R^{2d}\pi_*\Q_{X})\) is a global section 
of the local system \(\mathrm R^{2d}\pi_*\Q_{X}\). If the cohomology 
class $\alpha_0$ at a closed point \(0\in S\) is algebraic, the 
variational Hodge conjecture predicts that it is the restriction of 
an algebraic cycle class \([Z]\in \mathrm H^{2d}(X; \Q)\). 
The Hodge conjecture implies this statement by Deligne's global invariant 
cycle theorem~\cite[Th\'eor\`eme~4.1.1]{Deligne}.

Bloch proved an algebraicity criterion under a \emph{semi-regularity}
hypothesis~\cite[Theorem~7.4]{Bloch}. Suppose that $\alpha_0$ is the class of
a local complete intersection $Z_0\subset X_0$ of codimension $d$.
Bloch defines a semi-regularity map 
\[
s:\Ho^1(Z_0,N_{Z_0/X_0})\rightarrow \Ho^{d+1}(X_0,\Omega_{X_0/\C}^{d-1})
\] 
and shows that if $s$ is injective, then the cohomology classes 
$\{\alpha_t\}_{t\in S}$ are algebraic. 
In this paper, we study a related lifting problem for automorphisms.


Let $S=\Spec\C[[t_1,\ldots,t_n]]$, let $0\in S$ be its closed point, and
let $\pi:X\rightarrow S$ be a smooth projective morphism with closed fibre
$X_0$. By~\cite[Proposition~8.9]{Katz} the continuous integrable connection
\[
\nabla\colon \Ho^m_{\DR}(X/S)\to \Ho^m_{\DR}(X/S)\otimes \widehat{\Omega}^1_{S/\C}
\]
has a functorial horizontal trivialization:
\[
\Ho^m_{\DR}(X/S)\simeq \Ho^m_{\DR}(X/S)^{\nabla}\otimes_{\C} \OO_S\simeq
\Ho^m_{\DR}(X_0/\C)\otimes_{\C}\OO_S.
\]
For an automorphism $g_0$ of $X_0/\C$, let
\[
\widehat h:\Ho^m_{\DR}(X/S)\rightarrow\Ho^m_{\DR}(X/S)
\]
be the horizontal endomorphism corresponding to $\Ho^m_{\DR}(g_0)\otimes 1$ 
under this trivialization. The endomorphism $\widehat h$ need not 
preserve the canonical Hodge filtration \(\F^p\Ho^m_{\DR}(X/S)\) 
on $\Ho^m_{\DR}(X/S)$ since this filtration is not flat.

\begin{defi}\label{def:hor-end}
With the above notation, we say that an automorphism $g_0$ is of 
\emph{$m$-Hodge type} on $X/S$ if
\[
\widehat h\bigl(\F^p\Ho^m_{\DR}(X/S)\bigr)
\subseteq\F^p\Ho^m_{\DR}(X/S)\qquad\text{for every }p\geq0.
\]
\end{defi}

If $g_0$ extends to an automorphism \(g\) of $X/S$, then it is of $m$-Hodge
type, since the induced endomorphism $\Ho^m_{\DR}(g)$, which identifies with 
$\widehat h$, preserves the Hodge filtration.  
In the other direction, we prove the following theorem.

\begin{theorem}\label{thm:auto-lift-char-0}
Let $S=\Spec\C[[t_1,\ldots,t_n]]$ and let $\pi:X\rightarrow S$ be a
smooth projective morphism with closed fibre $X_0$. Suppose that for some
non-negative integer $m$ the infinitesimal Torelli map (with $\Omega^j=0$ for $j<0$)
\begin{equation}\label{introcup}
\Ho^1(X_0,T_{X_0})\rightarrow
\bigoplus_{p+q=m}\mathrm{Hom}\bigl(\Ho^q(X_0,\Omega^p_{X_0}),
\Ho^{q+1}(X_0,\Omega^{p-1}_{X_0})\bigr)
\end{equation}
is injective. If $g_0\in\Aut(X_0)$ is of $m$-Hodge type on $X/S$,
then $g_0$ extends to an $S$-automorphism $g:X\rightarrow X$. Its action
on $\Ho^m_{\DR}(X/S)$ is the horizontal endomorphism $\widehat h$.
\end{theorem}

For a smooth projective family $\pi:X\rightarrow S$ over a smooth complex
variety and a point $0\in S(\C)$, we may say that $g_0\in\Aut(X_0)$ is of
$m$-Hodge type on \(X/S\) if it is for the family obtained by base change to
the completion $\Spec\widehat{\OO}_{S,0}$.

\begin{cor}\label{cor:etale-lift-char-0}
If $X_0$ satisfies the infinitesimal Torelli~\eqref{introcup} 
and $g_0\in\Aut(X_0)$ is of $m$-Hodge type on $X/S$, then there are 
an \'etale neighborhood $(U,u)\rightarrow(S,0)$ and a $U$-automorphism
$g:X\times_S U\rightarrow X\times_S U$ restricting to $g_0$ over $u$.
After this \'etale base change, the action of $g$ on \(\mathrm H^{m}_{\DR}\) 
is the horizontal endomorphism induced by $g_0$.
\end{cor}

Theorem~\ref{thm:auto-lift-char-0} can be used to study the
faithfulness of the action of automorphisms on cohomology. 
For smooth projective complex curves of genus at least two 
and complex $K3$ surfaces, the action on their own middle 
cohomology groups is faithful; see~\cite[Lemma~1.14]{DM}
and~\cite[Proposition~1.1]{BurnsRapoport}.
Corollary~\ref{faithful} shows that faithfulness on a dense
open subset of the base specializes to any closed fiber satisfying
infinitesimal Torelli. This method applies to the cases of complete
intersections~\cite{PAN2} and cyclic coverings~\cite{PAN3}.

In the second part of the paper, we study the analogous lifting problem
for automorphisms of a variety in positive characteristic.

The $p$-adic variational Hodge conjecture, addressed by Maulik and 
Poonen~\cite[{\S} 9.2]{MP}, motivates the mixed-characteristic
analogue of Theorem~\ref{thm:auto-lift-char-0}. Let $W=W(k)$ be the
ring of Witt vectors of a perfect field \(k\) of characteristic $p>0$, and 
\(K\) the fraction field of \(W\). Let $X$ be a smooth proper $W$-scheme
with special fiber $X_0$ and generic fiber $X_K$. There is a
crystalline--de Rham comparison isomorphism
\[
  \sigma_\mathrm{cris}\colon \Ho^{2r}_{\DR}(X_K/K)\simeq \cris^{2r}(X_0/W)\otimes_W K 
\]
\begin{conjecture}[{\cite[Conjecture~9.2]{MP}}]
Let $Z_0\in\operatorname{CH}^r(X_0)_{\mathbb Q}$. The $p$-adic variational Hodge conjecture predicts that if its
crystalline cycle class 
\(\operatorname{cl}_{\mathrm{cris}}(Z_0)\in \cris^{2r}(X_0/W)\otimes_W K\) satisfies
\[
 \sigma_{\mathrm{cris}}^{-1}\bigl(\operatorname{cl}_{\mathrm{cris}}(Z_0)\bigr)
 \in \F^r\Ho^{2r}_{\DR}(X_K/K),
\]
there exists an algebraic cycle $Z_K\in\operatorname{CH}^r(X_K)_{\mathbb Q}$
such that
\[
\operatorname{cl}_{\DR}(Z_K)= \sigma_{\mathrm{cris}}^{-1}(\operatorname{cl}_{\mathrm{cris}}(Z_0)).
\]
\end{conjecture}

Inspired by the conjecture above and Definition~\ref{def:hor-end}, we aim to 
give a criterion for the lifting of an automorphism \(g_0\) on the 
closed fiber \(X_0\) in terms of the Hodge filtrations on the de Rham 
cohomology \(\Ho^*_{\DR}(X/W)\).

\begin{defi}\label{hodgetype}
Let $k$ be a perfect field of characteristic $p>0$, and let $X$ be smooth
and projective over $W=W(k)$ with special fiber $X_0$. Assume that the
Hodge--de Rham spectral sequence of $X/W$ degenerates at 
$E_1$ and that each $\Ho^j(X,\Omega^i_{X/W(k)})$ is a free $W$-module. 
We say that an automorphism $g_0$ of $X_0/k$ is of \emph{$m$-Hodge type} 
on $X/W$, for a given $m\geq0$, if the crystalline isomorphism
\[
 \cris^m(g_0):\cris^m(X_0/W)\rightarrow\cris^m(X_0/W)
\]
preserves the Hodge filtration transported by the integral
crystalline--de Rham comparison
\[
 \Ho^m_{\DR}(X/W)\simeq\cris^m(X_0/W).
\]
\end{defi}

We obtain the following $p$-adic analogue of 
Theorem~\ref{thm:auto-lift-char-0} in odd characteristic.

\begin{theorem}\label{corolifting}
Let $k$ be a perfect field of characteristic $p>2$, and let $X$ be smooth
and projective over $W(k)$ with special fiber $X_0$. Assume that the
Hodge--de Rham spectral sequence of $X/W(k)$ degenerates at $E_1$ and
that each $\Ho^j(X,\Omega^i_{X/W(k)})$ is free over $W(k)$.
Fix $m\geq0$ and suppose that the cup product
\begin{equation}
\label{cupproduct}
\Ho^1(X_0,T_{X_0})\rightarrow
\bigoplus\limits_{i+j=m} \mathrm{Hom}(\Ho^{j}(X_0,\Omega^i_{X_0/k}),\Ho^{j+1}(X_0,\Omega^{i-1}_{X_0/k}))
\end{equation}
is injective. If $g_0$ is an automorphism of $X_0$ of $m$-Hodge type
on $X/W(k)$ in the sense of Definition~\ref{hodgetype}, then $g_0$
lifts to an automorphism $g:X\rightarrow X$ over $W(k)$.
\end{theorem}

Esnault and Oguiso discussed the $p$-adic variational Hodge conjecture for
graphs of $K3$ automorphisms~\cite[Remark~6.5(1)]{EO}. The corresponding
lifting criterion for $K3$ surfaces appears in the proof of
\cite[Corollary~2.5]{Ogus}; see also~\cite[Remark~3.23]{derham}.
For abelian schemes over $W(k)$ with $p>2$, Berthelot and Ogus characterize
which homomorphisms of the special fibers lift by the Hodge filtration
on degree-one crystalline cohomology~\cite[Theorems~3.15 and~3.17]{derham}.
In Proposition~\ref{prop:rho-ob-commute}, we generalize their method 
by using divided power filtrations and Hodge filtrations on the
crystalline cohomology, and shows that preserving Hodge filtrations 
measures the vanishing of the deformation obstruction class of \(g_0\).
This leads to the proof of Theorem~\ref{corolifting}.

We also apply Theorem~\ref{corolifting} to the faithfulness of the action on
\'etale cohomology; see Corollary~\ref{padicfaithful}.

\textbf{Organization.} 
Section~\ref{homoalg} reviews the Kodaira--Spencer class and (continuous) Gauss--Manin connections. 
In Section~\ref{deforms}, we study deformation obstruction of automorphisms, and prove Theorem~\ref{thm:auto-lift-char-0}. 
Section~\ref{cryobs} Using divided power filtrations and Hodge filtrations on crystalline cohomology to detect the obstruction class, 
and complete the proof of Theorem~\ref{corolifting} in Section ~\ref{bloch}

\textbf{Acknowledgments.} The authors are grateful for S.~Bloch and H.~Esnault for comments on this project and pointing out references. The authors thank L.~Illusie and M.~Kerr for their suggestions. The second author thanks J.~de Jong for his lectures on crystalline cohomology and Y.~Tian and W.~Zheng for inviting him to the Morningside Center of Mathematics, where the first draft was written.

\textbf{AI disclosure.} The authors used GPT-6 to improve the proof methods and the organization of the first draft, check the mathematical conclusions, and refine the language. The authors have reviewed and verified the content and take full responsibility for its correctness.

\section{Connections on de Rham cohomology} \label{homoalg}
In this section, we review results of deformation theory and de Rham cohomology, 
which are used in Section \ref{deforms} to prove Theorem \ref{thm:auto-lift-char-0}.

Let $A$ be a $\C$-algebra with an ideal $I$, and let $\pi:X\rightarrow S$ be a smooth projective morphism where $S=\Spec (A)$. Denote by $S_0$ the affine scheme $\Spec (A_0)$ where $A_0=A/I$. Suppose that $X_0$ is the pullback of $X/S$ along $S_0\hookrightarrow S$. Denote the natural immersion $X_0\hookrightarrow X$ by $i$.
\[
\begin{tikzcd}
X_0 \arrow[r, hook, "i"] \arrow[d, "\pi_0"']
 \arrow[dr, phantom, "\Box" description] & X \arrow[d, "\pi"] \\
S_0 \arrow[r, hook] & S
\end{tikzcd}
\]
Let  $\kappa_{X/S}\in \Ext^1_X(\Omega^1_{X/S},\Omega^1_{S/\C}\otimes \OO_X)$
be the Kodaira--Spencer class of $X/S$, i.e., the class of the extension
\[
0\rightarrow \Omega^1_{S/\C}\otimes \OO_X\rightarrow \Omega^1_{X/\C}\rightarrow \Omega^1_{X/S}\rightarrow 0.
\] 
We denote by $\beta\in \Ext^1_{X_0}(\Omega^1_{X_0/S_0},\Omega^1_{S/\C}\otimes \OO_{X_0})$ the restriction of $\kappa_{X/S}$ to $X_0$, represented by
\begin{equation} 
\label{eq:KS-ext}
0\rightarrow \Omega^1_{S/\C}\otimes \OO_{X_0}\rightarrow \Omega^1_{X/\C}|_{X_0}\rightarrow \Omega^1_{X_0/S_0}\rightarrow 0.
\end{equation}
We write $\Ho^q_{\DR}(X/S)=\mathrm R^q\pi_*\Omega^\bullet_{X/S}$ for the relative de Rham cohomology sheaf and, when $S$ is affine, use the same notation for its module of global sections.
\begin{prop}[{\cite[Theorem~3.2]{Bloch}, \cite{Del}}]\label{prop:bc-deRham} Let $S$ be a scheme over $\Spec (\mathbb{Q})$, and let $\pi:X\rightarrow S$ be a proper and smooth morphism. Then
\begin{enumerate}
\item The sheaves $\mathrm{R}^q \pi_*(\Omega^p_{X/S})$ are locally free of finite type and commute with base change.
\item\label{item:hodge-degeneration} The spectral sequence 
\[E^{p,q}_1=\mathrm{R}^q \pi_*(\Omega^p_{X/S})\Longrightarrow \Ho^{p+q}_{\DR}(X/S)\]
degenerates at $E_1$.

\item The sheaves $\Ho^q_{\DR}(X/S)$ are locally free of finite type and commute with base change.
\end{enumerate}
\end{prop}

Let $S$ be a smooth $\mathbb{C}$-scheme. The canonical Gauss--Manin connection
\[
\nabla :\Ho^m_{\DR}(X/S)\rightarrow \Ho^m_{\DR}(X/S)\otimes_{\OO_S}\Omega^1_{S/\C}.
\]
satisfies the Griffiths transversality
\[
\nabla(\F^p\Ho^m_{\DR}(X/S))\subseteq \F^{p-1}\Ho^m_{\DR}(X/S)\otimes_{\OO_S}\Omega^1_{S/\C}.
\]
The induced map on the graded pieces is given by cup product with 
the Kodaira--Spencer class $\kappa_{X/S}$~\cite[Proposition~3.6]{Bloch}; 
see also~\cite{Griff}. Writing $\F^p=\F^p\Ho^m_{\DR}(X/S)$, one has the
commutative diagram
\[
\begin{tikzcd}[column sep=large]
\F^p/\F^{p+1} \arrow[r, "\operatorname{gr}^{p}_{\F}\nabla"] \arrow[d, equal] &
(\F^{p-1}/\F^p)\otimes\Omega^1_{S/\C} \arrow[d, equal] \\
\mathrm{R}^{m-p}\pi_*(\Omega^p_{X/S}) \arrow[r, "\kappa_{X/S}\cup-"] &
\mathrm{R}^{m-p+1}\pi_*(\Omega^{p-1}_{X/S})\otimes\Omega^1_{S/\C}
\end{tikzcd}.
\]

Now let \(S=\Spec A\), where $A=\C[[t_1,\ldots,t_n]]$ is a formal power
series ring, and let $\mathfrak m=(t_1,\ldots,t_n)$ be the maximal ideal. 
Let \(\widehat\Omega^1_{A/\C}\) be the \(\mathfrak m\)-adic completion of
the module of K\"ahler differentials. The spectral sequence of the
derived functor \(\mathrm R\pi_*\) applied to the completion of the
continuous absolute de Rham complex \(\widehat{\Omega}^\bullet_{X/\C}\) gives the continuous
integrable connection
\begin{equation}
\label{eq:cont-int-conn}
\widehat\nabla:\Ho^m_{\DR}(X/S)\rightarrow
\Ho^m_{\DR}(X/S)\otimes_A\widehat\Omega^1_{A/\C};
\end{equation}
see~\cite[III, \S4]{HartshorneDR}. Katz established the following
horizontal trivialization.

\begin{prop}[{\cite[Proposition~8.9]{Katz}}]
\label{prop:horizon-trivial}
Let \(k\) be a field of characteristic zero. Let $M$ be a finite 
$k[[t_1,\ldots,t_r]]$-module with a continuous integrable connection 
$\nabla$. This integrable connection is trivializable:
\[
(M, \nabla)\cong (M^{\nabla}\otimes_{k} k[[t_1,\ldots,t_r]], 1\otimes d).
\]
In particular, for the de Rham cohomology with the continuous integrable
connection~\eqref{eq:cont-int-conn}, we have the functorial horizontal 
trivialization
\[
\Ho^q_{\DR}(X/S)\simeq \Ho^q(X_0,\C)\otimes_{\C}A .
\]
\end{prop}

Write $\widehat\Omega^p_{X/S}$ and $\widehat\Omega^p_{X/\C}$ for the
\(\mathfrak m\)-adic completions of the relative differential forms and
the continuous absolute differential forms, which are coherent sheaves on the formal scheme 
\(\widehat X\). We denote by $\widehat\kappa$ the extension class of 
the formal cotangent sequence
\[
0\rightarrow\widehat\Omega^1_{A/\C}\otimes_A\OO_{\widehat X}
\rightarrow\widehat\Omega^1_{X/\C}
\rightarrow\widehat\Omega^1_{X/S}\rightarrow0.
\]
Since $X$ is proper over the complete ring $A$, formal functions
~\cite[III, Theorem~11.1]{H} identifies the relevant extension groups
by the local freeness of $\Omega^1_{X/S}$, so
\(\widehat \kappa\) can be regarded as an extension class in 
\(\operatorname{Ext}^1_X(\Omega^1_{X/S},
\widehat\Omega^1_{A/\mathbb C}\otimes_A\mathcal O_X)\).
The induced cup product map on the Hodge cohomology is
\[
\widehat\kappa\cup-:
\Ho^{q}(X,\Omega^p_{X/S})\rightarrow
\Ho^{q+1}(X,\Omega^{p-1}_{X/S})
\otimes_A\widehat\Omega^1_{A/\C}.
\]

\begin{prop}\label{prop:Griff-trans}
Write $\F^p=\F^p\Ho^m_{\DR}(X/S)$. With the notation above, 
the continuous integrable connection satisfies
Griffiths transversality:
\[
\widehat\nabla\bigl(\F^p\bigr)
\subseteq\F^{p-1}\otimes_A\widehat\Omega^1_{A/\C}.
\]
Its induced map on the Hodge graded pieces is given by
cup product with $\widehat\kappa$:
\[
\begin{tikzcd}[column sep=large]
\F^p/\F^{p+1}
  \arrow[r, "\operatorname{gr}^p_{\F}\widehat\nabla"]
  \arrow[d, equal]
&
(\F^{p-1}/\F^p)\otimes_A\widehat\Omega^1_{A/\C}
  \arrow[d, equal]
\\
\Ho^{m-p}(X,\Omega^p_{X/S})
  \arrow[r, "\widehat\kappa\cup-"]
&
\Ho^{m-p+1}(X,\Omega^{p-1}_{X/S})
  \otimes_A\widehat\Omega^1_{A/\C}.
\end{tikzcd}
\]
\end{prop}

\begin{proof}
The treatment is analogous to the algebraic Gauss-Manin connection.
The Griffiths transversality follows from the construction of 
$\widehat\nabla$ as a connecting morphism induced by the exterior 
powers of the formal cotangent sequence. On the Hodge graded pieces, the connecting morphism is the boundary of the
exterior cotangent extension, hence cup--contraction with $\widehat\kappa$.
This is the continuous analogue of the filtered-complex construction in
\cite[III, \S4]{HartshorneDR} and
\cite[Theorem~3.5.2]{Katz}.
\end{proof}

\section{Deformations of automorphisms} \label{deforms}




\begin{lemma}
\label{lem:obs-auto}
Let $A\rightarrow A_0$ be a surjective square-zero extension of 
$\C$-algebras with kernel $\II$, put $S=\Spec(A)$ and $S_0=\Spec(A_0)$. 
Suppose that the scheme $X$ is smooth over \(S\). For every 
$S_0$-automorphism $g_0$ of $X_0=X\times_S S_0$, there is a canonical obstruction
\[
 \operatorname{ob}(g_0)\in
 \Ext^1_{X_0}\left(g_0^*\Omega^1_{X_0/S_0},
                         \II\otimes_{A_0}\OO_{X_0}\right).
\]
It vanishes if and only if $g_0$ lifts to an $S$-automorphism of $X$.
Moreover, let
\[
d:\II\rightarrow\Omega^1_{S/\C}\otimes_A A_0
\]
be the map induced by the universal derivation, and let $\beta$ be the restricted Kodaira--Spencer
class~\eqref{eq:KS-ext}. Pullback along the cotangent isomorphism $dg_0$
induces $(dg_0)^*$ on extension classes. With this notation,
\begin{equation}
\label{eq:ob-KS}
 (d\otimes 1)_*\operatorname{ob}(g_0)=g_0^*\beta-(dg_0)^*\beta.
\end{equation}
\end{lemma}

\begin{proof}
The ideal \(\mathcal I\) is also an \(A_0\)-module since \(\II^2=0\). 
The ideal of $X_0$ in $X$ is \(\II\OO_{X}=\II\otimes_{A_0}\OO_{X_0}\).
Let $\{\widetilde U_i\}$ be an affine open cover of $X$ and put 
$U_i=\widetilde U_i\times_S S_0$. By the formal smoothness of 
$\pi:X\to S$, there exists an $S$-morphism $\widetilde g_i$ that lifts
\(g_0|_{U_i}\) and fits into the diagram
\[
\begin{tikzcd}
U_i \arrow[r, "i\circ g_0|_{U_i}"] \arrow[d, hook] &
X \arrow[d, "\pi"] \\
\widetilde U_i \arrow[r] \arrow[ur, dashed, "\widetilde g_i"'] & S
\end{tikzcd}
\]
It induces a morphism of sheaves of rings on $\widetilde U_i$,
\[
\varphi_i=\widetilde g_i^\sharp:
   \widetilde g_i^{-1}\OO_X\rightarrow\OO_{\widetilde U_i}.
\]
The underlying continuous maps of the $\widetilde g_i$ agree with $g_0$ on \(|\widetilde U_i|=|U_i|\). Therefore on overlaps 
\(\widetilde U_{ij} = \widetilde U_i\cap \widetilde U_j\), the 
difference \(D_{ij}=\varphi_i-\varphi_j\) defines an \(S\)-derivation
\[
    D_{ij}\colon \widetilde g_i^{-1}\OO_{X}|_{\widetilde U_{ij}}\rightarrow \II\OO_{X}|_{\widetilde U_{ij}}.
\]
For $a\in\II$ and a local section $b$ of $\OO_X$, one has
\(D_{ij}(ab)=aD_{ij}(b)\in \II^2\). Then the derivation \(D_{ij}\) 
descends to \(\OO_{X_0}\):
\[
D_{ij}:g_0^{-1}\OO_{X_0}|_{U_{ij}}\rightarrow
       (\II\otimes_{A_0}\OO_{X_0})|_{U_{ij}}.
\]
Thus $D_{ij}$ is an $S_0$-derivation along $g_0$, corresponding to
\[
\delta_{ij}:g_0^*\Omega^1_{X_0/S_0}|_{U_{ij}}
 \rightarrow(\II\otimes_{A_0}\OO_{X_0})|_{U_{ij}}.
\]
The \v{C}ech $1$-cocycle \((\delta_{ij})\) induces an obstruction class,
\[
\begin{aligned}
\operatorname{ob}(g_0)=[\delta_{ij}]
&\in \mathrm H^1\!\left(X_0,\mathcal Hom\!\left(
     g_0^*\Omega^1_{X_0/S_0},\II\otimes_{A_0}\OO_{X_0}\right)\right)\\
&\cong\Ext^1_{X_0}\!\left(
     g_0^*\Omega^1_{X_0/S_0},\II\otimes_{A_0}\OO_{X_0}\right).
\end{aligned}
\]
This class \(\operatorname{ob}(g_0)\) is independent of the liftings 
\(\widetilde{g}_i\). If it vanishes, one can modify maps \(\varphi_i\)
by local derivations 
\(\eta_i\colon \widetilde g_i^{-1}\OO_X|_{\widetilde U_i}\rightarrow\II \OO_{X}|_{\widetilde U_i}\)
whose induced derivations on $X_0$ differ by $D_{ij}$ and the ring homomorphisms 
\(\varphi_i-\eta_i\) glue to an $S$-morphism $\widetilde g:X\to X$. 
Conversely, a global lift gives the zero class. The resulting 
morphism is an automorphism. In fact, the lifting $\widetilde g$ 
induces the following commutative diagram with exact rows
\[
\begin{tikzcd}[column sep=small, row sep=large]
0 \arrow[r] & \II\otimes_{A_0}g_0^{-1}\OO_{X_0} 
\arrow[r] \arrow[d, "1_{\II}\otimes g_0^\sharp"'] &
\widetilde g^{-1}\OO_X \arrow[r] \arrow[d, "\widetilde g^\sharp"] &
i_*g_0^{-1}\OO_{X_0} \arrow[r] \arrow[d, "g_0^\sharp"] & 0 \\
0 \arrow[r] & \II\otimes_{A_0}\OO_{X_0} \arrow[r] &
\OO_X \arrow[r] & i_*\OO_{X_0} \arrow[r] & 0
\end{tikzcd}
\]
The outer vertical maps are isomorphisms. It follows that 
$\widetilde g^\sharp$ is an isomorphism.

Now let us compare the obstruction \(\operatorname{ob}(g_0)\) 
with the Kodaira--Spencer class $\beta$. We denote by 
\[
    \theta_i\colon \bigl(g_0^*(\Omega^1_{X/\C}|_{X_0})\bigr)|_{U_i}\to \Omega^1_{X/\C}|_{U_i}
\]
the restriction of $d\varphi_i$ to \(U_i\). There is the following
commutative diagram on \(U_i\) induced by the extension~\eqref{eq:KS-ext}
(we omit the restrictions to \(U_i\)):
\begin{equation}
\label{eq:ext-class-diagram}
\begin{tikzcd}[column sep=small, row sep=large]
0 \arrow[r] &
\Omega^1_{S/\C}\otimes\OO_{X_0}
 \arrow[r, "g_0^*\iota"] \arrow[d, equal] &
g_0^*(\Omega^1_{X/\C}|_{X_0})
 \arrow[r, "g_0^*q"] \arrow[d, "\theta_i"] &
g_0^*\Omega^1_{X_0/S_0} \arrow[r] \arrow[d, equal] & 0 \\
0 \arrow[r] &
\Omega^1_{S/\C}\otimes\OO_{X_0}
 \arrow[r, "\iota"]&
\Omega^1_{X/\C}|_{X_0}
 \arrow[r, "(dg_0)^{-1}q"] &
g_0^*\Omega^1_{X_0/S_0} \arrow[r]& 0.
\end{tikzcd}
\end{equation}
The first row is the extension representing the class $g_0^*\beta$. 
The second row is the pullback of the extension~\eqref{eq:KS-ext} along 
\(dg_0\colon g_0^*\Omega^1_{X_0/S_0}\to \Omega^1_{X_0/S_0}\), which 
represents the class \((dg_0)^*\beta\). Since \(dg_0\) is an isomorphism, 
the middle term of the second row is again \(\Omega^1_{X/\C}|_{X_0}\) 
and the quotient map is $(dg_0)^{-1}q$. This diagram commutes since 
\(\widetilde g_i\) is an \(S\)-morphism.

Since $\theta_i$ and $\theta_j$ induce the same maps on the end terms, their difference factors uniquely through a map
\[
g_0^*\Omega^1_{X_0/S_0}|_{U_{ij}}\rightarrow \Omega^1_{S/\C}\otimes\OO_{X_0}|_{U_{ij}}
\]
on the overlap \(U_{ij}=U_i\cap U_j\). The differentiation of the 
\(S_0\)-derivation \(\varphi_i-\varphi_j\) factors through 
\(\delta_{ij}\) and the derivation map
\[
d\otimes 1\colon \II\otimes_{A_0}\OO_{X_0}\to \Omega^1_{S/\C}\otimes\OO_{X_0},
\]
which yields the relation
\begin{equation}
\label{eq:diff-obs}
\theta_i-\theta_j
 =\iota\circ(d\otimes1)\circ\delta_{ij}\circ g_0^*q.
\end{equation}
The \v{C}ech cocycle \(\{(d\otimes1)\circ \delta_{ij}\}\) represents
the pushout extension class \((d\otimes1)_*\operatorname{ob}(g_0)\).

After refining the cover, choose local splittings $s_i$ of $g_0^*q$.
The difference \(s_j-s_i\) takes values in the kernel of $g_0^*q$, so there is a unique
\[
    a_{ij}\in \Gamma\left(U_{ij}, \mathcal Hom\left(g_0^*\Omega^1_{X_0/S_0},\Omega^1_{S/\C}\otimes\OO_{X_0}\right)\right)
\] 
such that \(g_0^*\iota\circ a_{ij}=s_j-s_i\). Then the cocycle \((a_{ij})\)
represents the class \(g_0^*\beta\). The maps $r_i=\theta_i s_i$ split the bottom row of~\eqref{eq:ext-class-diagram}. The cocycle \((b_{ij})\) determined by \(r_j-r_i=\iota\circ b_{ij}\) represents 
\((dg_0)^*\beta\). Using~\eqref{eq:diff-obs} and the commutativity in 
the diagram~\eqref{eq:ext-class-diagram}, we obtain
\[
\begin{aligned}
\iota\circ b_{ij}=r_j-r_i
 &=\theta_j(s_j-s_i)+(\theta_j-\theta_i)s_i\\
 &=\theta_jg_0^*\iota(a_{ij})-\iota\circ (d\otimes1)\circ \delta_{ij}\\
 &=\iota\circ a_{ij}-\iota\circ (d\otimes1)\circ \delta_{ij}.
\end{aligned}
\]
Since \(\iota\) is injective, it follows that 
\[
(d\otimes1)_*\operatorname{ob}(g_0)
 =g_0^*\beta-(dg_0)^*\beta
\]
in $\Ext^1_{X_0}(g_0^*\Omega^1_{X_0/S_0},
\Omega^1_{S/\C}\otimes\OO_{X_0})$.
\end{proof}




Now we are able to prove Theorem \ref{thm:auto-lift-char-0}.

\begin{proof}
Write $A=\C[[t_1,\ldots,t_n]]$ and $\mathfrak m=(t_1,\ldots,t_n)$.
If $n=0$, the assertion is immediate. For $N\geq0$, set
\[
A_N=A/\mathfrak m^{N+1},\qquad X_N=X\times_S\Spec A_N,\qquad
I_{N+1}=\mathfrak m^{N+1}/\mathfrak m^{N+2}.
\]
Suppose that an automorphism $g_N$ of $X_N/A_N$ lifting $g_0$ has been
constructed. Consider the square-zero extension \(A_{N+1}\rightarrow A_N\)
with kernel \(I_{N+1}\). By Lemma~\ref{lem:obs-auto}, the obstruction 
to lifting \(g_N\) to an automorphism on the deformation \(X_{N+1}/A_{N+1}\) 
gives a class
\[
\operatorname{ob}(g_N)\in
\operatorname{Ext}^1_{X_N}
\bigl(
g_N^*\Omega^1_{X_N/A_N},
I_{N+1}\otimes_{A_N}\mathcal O_{X_N}
\bigr).
\]
Restricting the Kodaira--Spencer class \(\kappa_{X_{N+1}/A_{N+1}}\) to \(X_N\) 
corresponds to the extension
\[
 0\rightarrow \Omega^1_{A_{N+1}/\C}\otimes \OO_{X_{N}}\rightarrow \Omega^1_{X_{N+1}/\mathbb C}|_{X_N}\rightarrow \Omega^1_{X_N/A_N}\rightarrow 0,
\]
cf.~\eqref{eq:KS-ext}. Note that
\[
\Omega^1_{A_{N+1}/\C}\otimes_{A_{N+1}}A_N
\cong\bigoplus_{j=1}^n A_N\,dt_j.
\]
This extension gives a class 
\[
\beta_N\in\Ext^1_{X_N}\Bigl(\Omega^1_{X_N/A_N},
\bigl(\textstyle\bigoplus_{j=1}^n A_N\,dt_j\bigr)\otimes_{A_N}\OO_{X_N}\Bigr).
\]
It follows from the formula~\eqref{eq:ob-KS} that
\begin{equation}
\label{eq:N-obs-ks}
    (d_N\otimes1)_*(\operatorname{ob}(g_N))= g_N^*\beta_N-(dg_N)^*\beta_N
\end{equation}
where \(d_N\) is the map induced by the derivation
\[
d_N:I_{N+1}\rightarrow
 \Omega^1_{A_{N+1}/\C}\otimes_{A_{N+1}}A_N
 \cong\bigoplus_{j=1}^n A_N\,dt_j.
\]
Euler's identity $\sum_j t_j\partial_j f=(N+1)f$ for a homogeneous polynomial $f$ of degree $N+1$ shows that $d_N$ is injective.

We claim that the cup product \(\operatorname{ob}(g_N)\cup_{g_N}-\) 
vanishes on \(\Ho^q(X_N,\Omega^p_{X_N/A_N})\) for $p+q=m$ and $p\geq1$. Let $\widehat h$ be 
the horizontal endomorphism in Definition~\ref{def:hor-end}, which
satisfies
\[
\widehat \nabla\circ\widehat h=(\widehat h\otimes1)\circ\widehat \nabla,
\qquad \widehat h(\F^p)\subseteq\F^p.
\]
By Proposition~\ref{prop:Griff-trans}, the compatibility of the 
graded continuous integrable connection with the cup product map 
of the formal Kodaira--Spencer class \(\widehat \kappa\) 
gives the commutative diagram
\begin{equation}
\label{diam:higgs-bundle-end}
\begin{tikzcd}[column sep=large, row sep=large]
\Ho^q(X,\Omega^p_{X/S})
  \arrow[r, "\widehat\kappa\cup-"]
  \arrow[d, "\operatorname{gr}_{\F}^{p}\widehat h"']
&
\Ho^{q+1}(X,\Omega^{p-1}_{X/S})
  \otimes_A\widehat\Omega^1_{A/\C}
  \arrow[d, "\operatorname{gr}_{\F}^{p-1}\widehat h\otimes 1"]
\\
\Ho^q(X,\Omega^p_{X/S})
  \arrow[r, "\widehat\kappa\cup-"]
&
\Ho^{q+1}(X,\Omega^{p-1}_{X/S})
  \otimes_A\widehat\Omega^1_{A/\C}.
\end{tikzcd}
\end{equation}
Let $h_N$ be the restriction of $\widehat h$ to $A_N$, and let
\[
g_N^*=\Ho^m_{\DR}(g_N):\Ho^m_{\DR}(X_N/A_N)
\rightarrow\Ho^m_{\DR}(X_N/A_N).
\]
By functoriality of the horizontal trivialization in
Proposition~\ref{prop:horizon-trivial}, both $h_N$ and $g_N^*$ correspond
to $\Ho^m_{\DR}(g_0)\otimes1$. Thus $g_N^*=h_N$, and in particular
\[
\operatorname{gr}_{\F}^{p}\Ho^m_{\DR}(g_N)
=\operatorname{gr}_{\F}^{p}h_N.
\]

Note that
\[
\widehat\Omega^1_{A/\C}\otimes_A A_N
\simeq \bigoplus_{i=1}^{n}A_{N}dt_i\simeq \Omega^1_{A_{N+1}/\C}\otimes_{A_{N+1}}A_N.
\]
W can assert that the restriction of the continuous cotangent sequence
defining $\widehat\kappa$ to $X_N$ is naturally isomorphic to the 
cotangent sequence defining $\beta_N$. Thus $\widehat\kappa|_{X_N}=\beta_N$.
By base change for Hodge cohomology, restriction of the horizontal arrows of~\eqref{diam:higgs-bundle-end} gives
\[
\beta_N\cup-:
 \Ho^q(X_N,\Omega^p_{X_N/A_N})\rightarrow
 \Ho^{q+1}(X_N,\Omega^{p-1}_{X_N/A_N})
 \otimes_{A_N}\bigl(\textstyle\bigoplus_{j=1}^n A_N\,dt_j\bigr).
\]
The commutativity of the diagram~\eqref{diam:higgs-bundle-end}
yields the equation
\[
g_N^*(\beta_N\cup a) = \beta_N\cup g_N^*(a), \quad \forall a\in \Ho^q(X_N,\Omega^p_{X_N/A_N}).
\]
By the description of the classes \(g_N^*\beta_N\) and \((dg_N)^*\beta_N\) 
in Lemma~\ref{lem:obs-auto}, there are identifications
\[
\begin{aligned}
 (g_N^*\beta_N)\cup_{g_N}a
     &=g_N^*(\beta_N\cup a),\\
 ((dg_N)^*\beta_N)\cup_{g_N}a
     &=\beta_N\cup g_N^*(a).
 \end{aligned}
\]
Here $\xi\cup_{g_N}a$ is obtained by pulling $a$ back with coefficients in $g_N^*\Omega^p_{X_N/A_N}$, contracting with $\xi$ to obtain coefficients in $g_N^*\Omega^{p-1}_{X_N/A_N}$ tensored with the coefficient module of $\xi$, and applying $\bigwedge^{p-1}dg_N$ to the resulting forms.
Pullback commutes with contraction, which gives the first equality. For the second, both contractions are represented by
\[
g_N^*\Omega^p_{X_N/A_N}\xrightarrow{\bigwedge^p dg_N} \Omega^p_{X_N/A_N}
\xrightarrow{\beta_N\lrcorner-}\Omega^{p-1}_{X_N/A_N}\otimes_{A_N}
(\bigoplus_{j=1}^n A_N\,dt_j)[1].
\]
Write \(\delta_N=g_N^*\beta_N-(dg_N)^*\beta_N\). The cup product with 
\(\delta_N\) along \(g_N\) vanishes on \(\Ho^q(X_N,\Omega^p_{X_N/A_N})\). 
It follows from~\eqref{eq:N-obs-ks} that
\begin{equation}
\label{eq:obs-vanish}
    (d_N\otimes1)_*(\operatorname{ob}(g_N)\cup_{g_N} a)=\delta_N\cup_{g_N} a=0, \quad \forall a\in \Ho^q(X_N,\Omega^p_{X_N/A_N}).
\end{equation}
By Proposition~\ref{prop:bc-deRham}, the Hodge cohomology 
 module $\mathrm H^{q+1}(X_N,\Omega^{p-1}_{X_N/A_N})$ is free over $A_N$, so tensoring $d_N$ with it preserves injectivity. Therefore 
\(\operatorname{ob}(g_N)\cup_{g_N} -\) vanishes on \(\Ho^q(X_N,\Omega^p_{X_N/A_N})\) as claimed.

We apply infinitesimal Torelli on \(X_0\) to prove that 
\(\operatorname{ob}(g_N)=0\). In fact \(\mathfrak m I_{N+1}=0\), so 
\(\Omega^{k}_{X_N/A_N}\otimes_{A_N}I_{N+1}\) is supported on \(X_0\). 
To be precise, 
\[
\Omega^{k}_{X_N/A_N}\otimes_{A_N}I_{N+1}\cong i_*(\Omega^{k}_{X_0/\mathbb C}\otimes_{\mathbb C}I_{N+1})
\]
where \(i\colon X_0\hookrightarrow X_N\) is the inclusion. By the local
freeness of \(\Omega^1_{X_N/A_N}\) and the relation \(i\circ g_0=g_N\circ i\), 
the obstruction \(\operatorname{ob}(g_N)\) can be regarded as an 
extension class in
\[
\operatorname{Ext}^1_{X_N}
\bigl(g_N^*\Omega^1_{X_N/A_N},I_{N+1}\otimes_{A_N}\mathcal O_{X_N}\bigr)\cong 
\mathrm H^1(X_0, g_0^*T_{X_0/\C})\otimes_{\C} I_{N+1}.    
\]
With the preceding description of cup product along \(g_N\) this gives the commutative diagram
\[
\begin{tikzcd}[column sep=large, row sep=large]
\Ho^q(X_N,\Omega^p_{X_N/A_N})
 \arrow[r, "\operatorname{ob}(g_N)\cup_{g_N}-"]
 \arrow[d, two heads, "\mathrm{res}"'] &
\Ho^{q+1}(X_N,\Omega^{p-1}_{X_N/A_N})\otimes_{A_N}I_{N+1} \\
\Ho^q(X_0,\Omega^p_{X_0/\C})
 \arrow[r, "\operatorname{ob}(g_N)\cup_{g_0}-"'] &
\Ho^{q+1}(X_0,\Omega^{p-1}_{X_0/\C})\otimes_{\C}I_{N+1}
 \arrow[u, equal]
\end{tikzcd}
\]
The left vertical arrow is surjective by Hodge base change. Hence $\operatorname{ob}(g_N)\cup_{g_0}-$ vanishes. The infinitesimal
Torelli map~\eqref{introcup} tensored over $\C$ with $I_{N+1}$ 
remains injective, which implies \(\operatorname{ob}(g_N)=0\).
Thus $g_N$ extends to $g_{N+1}$, and induction gives a compatible 
system $(g_N)_{N\geq0}$.

The compatible system $(g_N)$ is a formal automorphism \(\widehat g\)
of the formal completion \(\widehat X\). Since $X$ is proper over the
complete Noetherian ring $A$, Grothendieck's existence theorem algebraizes 
\(\widehat g\) to a unique $S$-automorphism $g:X\rightarrow X$. By functoriality of the horizontal trivialization, its action on 
$\Ho^m_{\DR}(X/S)$ is $\widehat h$.
\end{proof}

\begin{proof}[Proof of Corollary~\ref{cor:etale-lift-char-0}]
Put $A=\widehat{\OO}_{S,0}$. By Theorem~\ref{thm:auto-lift-char-0},
$g_0$ extends to an automorphism over $\Spec A$. Shrink $S$ to an affine
neighborhood of $0$ and choose an affine open $V$ containing $g_0$ in the
relative isomorphism scheme $\operatorname{Isom}_S(X,X)$. The extension
defines an $A$-point of $V$. Artin approximation~\cite[Theorem~1.10]{Artin} 
gives an $\OO^h_{S,0}$-point of $V$ with reduction $g_0$. Then there exist
an \'etale neighborhood \((U,u)\) over \((S,0)\) and a $U$-automorphism 
$g$ of $X\times_S U$ extending $g_0$. Functoriality and uniqueness of 
the horizontal extension show that the action of $g$ is the one induced 
by $g_0$.
\end{proof}

\begin{cor}\label{faithful}
Let $S$ be a smooth, connected complex variety and let $X_0$ be the 
fibre over a complex point $0\in S$ of a smooth projective morphism 
$\pi:X\to S$. Assume that \(X_0\) satisfies the infinitesimal Torelli condition
for some positive integer $m$. If the kernel of 
$\Aut(X_s)\to\Aut(\mathrm{H}^m(X_s,\Q))$ is trivial for every complex 
point $s$ in a dense open subset of $S$, then the corresponding kernel 
for $X_0$ is trivial.
\end{cor}
\begin{proof}
Let $S^\circ\subset S$ be the dense open subset in the hypothesis and 
let $f_0\in\Aut(X_0)$ act trivially on $\Ho^m(X_0,\Q)$, so it is of 
$m$-Hodge type. Corollary~\ref{cor:etale-lift-char-0} gives a pointed 
connected \'etale neighborhood $(U,u)\to(S,0)$ and an automorphism $f$ of $X_U/U$ extending $f_0$. Its action on the degree-$m$ Betti local system is the identity at $u$, hence is the identity on $U$. The preimage $U^\circ$ of $S^\circ$ is dense in $U$ since $U\to S$ is open. Faithfulness gives $f_s=\Id$ for every complex point $s\in U^\circ$. Hence \(f\) and the identity map
agree on the dense open subset \(X_{U^\circ}\).
It follows that $f=\Id_{X_U}$ and hence $f_0=\Id_{X_0}$.
\end{proof}

\section{Transversality in crystalline cohomology} \label{cryobs}

\begin{Assumption}
\label{assumption}
Let $S_0\hookrightarrow S$ be a closed immersion of affine schemes 
defined by a square-zero ideal $\II$, and suppose that a prime $p$ is
nilpotent on $S$. We fix a divided power structure $\gamma$ on $\II$
satisfying \(\II^{[a]}=0, a\ge 2\).

Let \(Z\) be a smooth projective $S$-scheme and put \(Z_0=Z\times_S S_0\).
Assume that the Hodge--de Rham spectral sequence of $Z/S$ degenerates at
$E_1$ and that its de Rham and Hodge cohomology groups are locally free
over $S$ and commute with base change.
\end{Assumption}

\subsection{Divided power and Hodge filtrations}

Consider the canonical projection of topoi:
\[
 u_{Z_0/S}:(Z_0/S,\gamma)_{\mathrm{cris}}
                         \rightarrow (Z_0)_{\mathrm{zar}}.
\]
On this crystalline site, let $\mathcal J_{Z_0/S}$ be the PD ideal
sheaf in the structure sheaf $\mathcal O_{Z_0/S}$; see
\cite[5.2(3)]{BO}.
Its divided powers give the filtration
\[
\OO_{Z_0/S}=\mathcal J^{[0]}_{Z_0/S} \supseteq \mathcal J^{[1]}_{Z_0/S} \supseteq \mathcal J^{[2]}_{Z_0/S} \supseteq\cdots
\]
Let $\F_Z^a\Omega^\bullet_{Z/S}=\Omega^{\bullet \geq a}_{Z/S}$ denote
the Hodge filtration on the de Rham complex. We write $\mathrm F_{Z_0}$
for the PD filtration; the underlying complex or cohomology group
specifies which filtration is meant.
For $a\geq 0$, the degree-$s$ term of the subcomplex
\(\mathrm{F}^a_{Z_0}\Omega_{Z/S}^{\bullet}\subseteq\Omega_{Z/S}^{\bullet}\) is
\[
 (\mathrm{F}^a_{Z_0}\Omega^\bullet_{Z/S})^s
   =\II^{[a-s]}\Omega^s_{Z/S},
 \quad s\ge 0 \text{~and~} \II^{[b]}=\OO_S, b\leq0.
\]
Equivalently, this subcomplex is
\[
\sum_{\substack{b+c=a\\b,c\geq0}}\II^{[b]}\F_Z^c\Omega^\bullet_{Z/S}= 
\bigl[0\to \II^{[a]}\OO_Z\to\cdots\to \II\Omega^{a-1}_{Z/S} \to
\Omega^{a}_{Z/S}\to \cdots\bigr].
\]
The filtered crystalline--de Rham comparison
\cite[Theorem 7.2]{BO}, applied to $Z_0\hookrightarrow Z$, gives a canonical
quasi-isomorphism (see also~\cite[Theorem 2.1]{Trans})
\begin{equation}
\label{eq:crys-pd-filtration}
\mathrm{R}u_{Z_0/S,*}\mathcal J_{Z_0/S}^{[a]}
\simeq \sum_{\substack{b+c=a\\b,c\geq0}}\II^{[b]}\F_Z^c\Omega^\bullet_{Z/S}
\end{equation}
on the common Zariski space of $Z_0$ and $Z$. Under the crystalline--de Rham
comparison, the induced filtration on cohomology is
\[
 \begin{aligned}
 \mathrm{F}_{Z_0}^a\Ho^m_{\DR}(Z/S)
 &:={\rm Im}\bigl(
 \Ho^m_{\mathrm{cris}}(Z_0/S,\mathcal J_{Z_0/S}^{[a]})
                   \rightarrow\cris^m(Z_0/S)\bigr)\\
 &=\operatorname{Im}\bigl(
 \mathbb H^m(Z,\mathrm{F}_{Z_0}^a\Omega_{Z/S}^{\bullet})\rightarrow\Ho^m_{\DR}(Z/S)\bigr).
 \end{aligned}
\]
\begin{lemma}
Suppose that the Hodge--de Rham spectral sequence for $Z/S$ degenerates
at $E_1$ and that the de Rham and Hodge cohomology are locally free and
commute with base change. Then
\[
 \mathrm{F}_{Z_0}^a\Ho^m_{\DR}(Z/S)= \sum_{\substack{b+c=a\\b,c\geq0}}
                   \II^{[b]}\F_Z^c\Ho^m_{\DR}(Z/S).
\]
\end{lemma}
\begin{proof}
See the proof of~\cite[Theorem 2.4, pp.~405--406]{Trans}. 
\end{proof}

\begin{remark}
Defining the PD filtrations \(\mathrm{F}^a_{Z_0}\) does not use the assumption
\(\II^{[a]}=0\) for \(a\geq 2\). In the rest of this section, we apply 
this additional condition. Then the complex and cohomology filtrations 
simplify to
\begin{equation}\label{eq:PD-first-order-fil}
 \mathrm{F}_{Z_0}^a\Omega^\bullet_{Z/S}
 =\F_Z^a\Omega^\bullet_{Z/S}
    +\II\F_Z^{a-1}\Omega^\bullet_{Z/S},
\end{equation}
\[
\mathrm{F}_{Z_0}^a\Ho^m_{\DR}(Z/S)=\F_Z^a\Ho^m_{\DR}(Z/S)
                      +\II\F_Z^{a-1}\Ho^m_{\DR}(Z/S).
\]
\end{remark}

Let $X$ and $Y$ be smooth projective $S$-schemes satisfying
Assumption~\ref{assumption}, with reductions $X_0=X\times_S S_0$
and $Y_0=Y\times_S S_0$, and let $f_0:X_0\rightarrow Y_0$ be an
$S_0$-morphism.
Let $j:X_0\hookrightarrow X$ and $i:Y_0\hookrightarrow Y$ denote
the closed immersions. The morphism of crystalline topoi
\[
f_{0,\mathrm{cris}}:(X_0/S,\gamma)_{\mathrm{cris}}
\to(Y_0/S,\gamma)_{\mathrm{cris}}
\]
induces the crystalline pullback 
\[
 \Phi^m=\cris^m(f_0):
 \cris^m(Y_0/S)\rightarrow\cris^m(X_0/S).
\]
By the canonical homomorphism
\[
  \mathrm{R}u_{Y_0/S,*}\mathcal J_{Y_0/S}^{[a]}\to \mathrm{R}{f_0}_*\mathrm{R}u_{X_0/S,*}\mathcal J_{X_0/S}^{[a]},
\]
the crystalline pullback \(\Phi^m\) preserves the PD filtrations
\[
 \Phi^m\bigl(\mathrm{F}_{Y_0}^a\Ho^m_{\DR}(Y/S)\bigr)
      \subseteq \mathrm{F}_{X_0}^a\Ho^m_{\DR}(X/S).
\]

By~\eqref{eq:PD-first-order-fil} the Hodge and PD 
filtrations satisfy the inclusion relations
\[
\mathrm{F}_{Z_0}^{k+1}\Ho^m_{\DR}(Z/S)  \subseteq \F_{Z}^k\Ho^m_{\DR}(Z/S)  \subseteq \mathrm{F}_{Z_0}^k\Ho^m_{\DR}(Z/S).  
\]
For $1\leq k\leq m$, restricting \(\Phi^m\) to
\(\F_{Y}^k\Ho^m_{\DR}(Y/S)\) gives the following diagram.
\begin{equation}\label{defrohf}
\begin{tikzcd}[column sep=2.3em,row sep=3em]
 \mathrm{F}_{Y_0}^{k+1}\Ho^m_{\DR}(Y/S)
   \ar[r, hook] \ar[d,"\Phi^m"']
 & \F_Y^k\Ho^m_{\DR}(Y/S)
   \ar[r] \ar[d,"{\left.\Phi^m\right|_{\F_Y^k}}"]
 & \operatorname{gr}_{\F_{Y_0}}^k\Ho^m_{\DR}(Y_0/S_0)
   \ar[d,dashed,"\rho(f_0)_k"] \\
 \F_X^k\Ho^m_{\DR}(X/S)
   \ar[r, hook]
 & \mathrm{F}_{X_0}^k\Ho^m_{\DR}(X/S)
   \ar[r]
 & \operatorname{gr}_{\F_X}^{k-1}\Ho^m_{\DR}(X/S)
       \otimes_{\OO_S}\II .
\end{tikzcd}
\end{equation}
Here $\F_{Y_0}$ denotes the Hodge filtration on $\Ho^m_{\DR}(Y_0/S_0)$
and the PD filtration on
$\Ho^m_{\DR}(Y/S)$. By $E_1$-degeneration and base change, the top
right arrow is surjective with kernel
\[
 \mathrm F_{Y_0}^{k+1}\Ho^m_{\DR}(Y/S)
 =\F_Y^{k+1}\Ho^m_{\DR}(Y/S)+\II\F_Y^k\Ho^m_{\DR}(Y/S).
\]
The left square commutes because $\Phi^m$ preserves the PD filtration
and $\mathrm F_{X_0}^{k+1}\subseteq\F_X^k$. The bottom right term is
the quotient of $\mathrm F_{X_0}^k\Ho^m_{\DR}(X/S)$ by
$\F_X^k\Ho^m_{\DR}(X/S)$, using local freeness of the Hodge graded
pieces. These quotient maps determine $\rho(f_0)_k$. Since Hodge cohomology
commutes with base change and $\II$ is an $\OO_{S_0}$-module, its target
identifies with 
\[
\operatorname{gr}_{\F_{X_0}}^{k-1}\Ho^m_{\DR}(X_0/S_0)\otimes_{\OO_{S_0}}\II\cong
\Ho^{m-k+1}(X_0,\Omega^{k-1}_{X_0/S_0})\otimes_{\OO_{S_0}}\II.
\]
Therefore we obtain the map
\begin{equation}
\label{eq:Hodge-PD-defect}
  \rho(f_0)_k\colon \Ho^{m-k}(Y_0,\Omega^k_{Y_0/S_0})\to 
  \Ho^{m-k+1}(X_0,\Omega^{k-1}_{X_0/S_0})\otimes_{\OO_{S_0}}\II.
\end{equation}
The obstruction to lifting \(f_0\) to an \(S\)-morphism \(X\to Y\) is a class
\[
  \operatorname{ob}(f_0)\in \mathrm{Ext}^1_{X_0}\bigl(f_0^*\Omega^1_{Y_0/S_0},
  \II\otimes_{\OO_{S_0}}\OO_{X_0}\bigr),
\]
which induces the cup product map
\begin{equation}
\label{eq:cup-prod-obs}
\operatorname{ob}(f_0)\cup_{f_0} -\colon \Ho^{m-k}(Y_0,\Omega^k_{Y_0/S_0}) \rightarrow  
\Ho^{m-k+1}(X_0,\Omega^{k-1}_{X_0/S_0})\otimes_{\OO_{S_0}}\II
\end{equation}

\begin{prop} 
\label{prop:rho-ob-commute}
Under the preceding hypotheses, including $\II^{[2]}=0$ for the fixed
PD structure, let $f_0:X_0\rightarrow Y_0$ be an $S_0$-isomorphism.
For $1\leq k\leq m$,
\[
 \rho(f_0)_k=(-1)^{m-k}\,\operatorname{ob}(f_0)\cup_{f_0} - .
\]
\end{prop}
Berthelot and Ogus prove the comparison for \(k=1\) in~\cite[Proposition~3.20 and Theorem~3.21]{derham}. 
Their result does not assume the vanishing \(\II ^{[2]}=0\).
We establish the comparison for $k\ge 1$ under the square-zero PD thickening hypothesis.

\textbf{Categorical description of \(\rho(f_0)\).}
The morphism $\rho(f_0)_k$ is the cohomological realization of
a morphism in the derived category \(D(\OO_{X_0})\) of
$\OO_{X_0}$-modules. 
For $Z=X$ or $Y$, the filtered comparison~\eqref{eq:crys-pd-filtration}
and $\II^{[2]}=0$ give an isomorphism in $D(\OO_{Z_0})$
\[
\begin{aligned}
\mathrm{R}u_{Z_0/S,*}\bigl(\mathcal J_{Z_0/S}^{[k]}/\mathcal J_{Z_0/S}^{[k+1]}\bigr)
&\simeq \operatorname{gr}_{\mathrm{F}_{Z_0}}^k\Omega^\bullet_{Z/S}\\
&\simeq\bigl(\Omega^{k-1}_{Z_0/S_0}\otimes_{\OO_{S_0}}\II\bigr)[1-k]
 \oplus\Omega^k_{Z_0/S_0}[-k].
\end{aligned}
\]
The two displayed summands have zero differential. The morphism
\(f_{0,\mathrm{cris}}\) induces the canonical morphism
\[
\Psi_k:\mathrm{R}u_{Y_0/S,*}\bigl(\mathcal J_{Y_0/S}^{[k]}/\mathcal J_{Y_0/S}^{[k+1]}\bigr)\rightarrow \mathrm{R}f_{0*}\mathrm{R}u_{X_0/S,*}\bigl(\mathcal J_{X_0/S}^{[k]}/\mathcal J_{X_0/S}^{[k+1]}\bigr).
\]
Let $\iota_{Y,k}$ and $\tau_{X,k}$ be the natural inclusion and
projection, respectively:
\begin{equation}
\label{eq:incl-proj}
\begin{aligned}
 \iota_{Y,k}:\Omega^k_{Y_0/S_0}[-k]
   &\rightarrow\operatorname{gr}_{\mathrm{F}_{Y_0}}^k\Omega^\bullet_{Y/S},\\
 \tau_{X,k}:\operatorname{gr}_{\mathrm{F}_{X_0}}^k\Omega^\bullet_{X/S}
   &\rightarrow\bigl(\Omega^{k-1}_{X_0/S_0}
                     \otimes_{\OO_{S_0}}\II\bigr)[1-k].
\end{aligned}
\end{equation}
Consider the composite
\[
\mathrm{R}f_{0*}\tau_{X,k}\circ \Psi_k\circ \iota_{Y,k}\colon \Omega^k_{Y_0/S_0}[-k]\rightarrow
\mathrm{R}f_{0*}\bigl((\Omega^{k-1}_{X_0/S_0}\otimes_{\OO_{S_0}}\II)[-k+1]\bigr).
\]
Derived adjunction gives a morphism in $D(\OO_{X_0})$. Since
$\Omega^k_{Y_0/S_0}$ is locally free, its derived pullback is its
ordinary pullback:
\begin{equation}\label{categoricalrhosymbol}
 \widetilde\rho_{k}(f_0):
 f_0^*\Omega^k_{Y_0/S_0}[-k]\rightarrow
          (\Omega^{k-1}_{X_0/S_0}\otimes_{\OO_{S_0}}\II)[-k+1].
\end{equation}
As $f_0$ is an isomorphism, pullback gives an identification
\[
\Ho^{m-k}(Y_0,\Omega^k_{Y_0/S_0})\simeq
\Ho^{m-k}(X_0,f_0^*\Omega^k_{Y_0/S_0}).
\]
Under this identification,
applying \(\Ho^m(\mathrm{R}\Gamma(X_0,-))\) gives $\rho(f_0)_k$.

\textbf{Cup product with \(\operatorname{ob}(f_0)\).} We use the following exterior map.
\begin{equation}\label{eq:inclusion}
 \begin{aligned}
 \operatorname{incl}:f_0^*\Omega^k_{Y_0/S_0}&\rightarrow f_0^*\Omega^1_{Y_0/S_0}\otimes_{\OO_{X_0}} f_0^*\Omega^{k-1}_{Y_0/S_0},\\
 v_1\wedge\cdots\wedge v_k&\longmapsto
 \sum_{a=1}^k(-1)^a v_a\otimes
 v_1\wedge\cdots\widehat v_a\cdots\wedge v_k .
 \end{aligned}
\end{equation}
With this sign convention, the cup product~\eqref{eq:cup-prod-obs}
is induced, after pullback along $f_0$, by the composite
\[
\begin{aligned}
 f_0^*\Omega^k_{Y_0/S_0}
 &\xrightarrow{\operatorname{incl}}
 f_0^*\Omega^1_{Y_0/S_0}\otimes_{\OO_{X_0}} f_0^*\Omega^{k-1}_{Y_0/S_0}\\
 &\xrightarrow{\operatorname{ob}(f_0)\otimes1}
 \II\otimes_{\OO_{S_0}} f_0^*\Omega^{k-1}_{Y_0/S_0}[1]\\
 &\xrightarrow{1\otimes\Lambda^{k-1}df_0}
 \II\otimes_{\OO_{S_0}}\Omega^{k-1}_{X_0/S_0}[1].
 \end{aligned}
\]

\subsection{The divided power envelope of the graph}
\label{sub:graph-embedding}
We regard $X_0$ as a closed subscheme of $X\times_S Y$ by the graph 
embedding
\[
  (j,i\circ f_0)\colon X_0\hookrightarrow X\times_S Y
\]
Put $Z=X\times_S Y$ and $Z_0=X_0\times_{S_0}Y_0$. Let 
$D=D_{X_0,\gamma}(Z)$ be its divided power envelope compatible 
with the fixed PD structure \(\gamma\) on $\II$, and let $\ojj$ be 
the PD ideal of $X_0$ in $D$.

The \emph{divided power de Rham complex} attached to the embedding 
$X_0\hookrightarrow Z$ is the complex 
\[
 \Omega^\bullet_{D/S}:=\OO_D\otimes_{\OO_Z}\Omega^\bullet_{Z/S}.
\]
The differential on \(\Omega^\bullet_{D/S}\) is induced by the 
canonical PD connection on $\OO_D$. For any \(a\in \ojj\),
\[
d(\gamma_n(a))=\gamma_{n-1}(a)da, \quad n\ge 1
\]
where \(\gamma_j\) is the \(j\)-th divided power of the fixed 
\(\gamma\). For any local sections \(a\in \OO_D\) and 
\(\omega\in \Omega_{Z/S}^{s}\), it satisfies Leibniz's rule
\[
d(a\otimes\omega)=da\wedge\omega+a\otimes d\omega.
\]
The divided power order defines a filtration
\[
 \mathrm{F}_{X_0}^k\Omega^s_{D/S}:=\ojj^{[k-s]}\Omega^s_{D/S},
 \qquad\ojj^{[n]}=\OO_D\quad(n\leq0).
\]
Since $\ojj\,\ojj^{[n]}\subseteq\ojj^{[n+1]}$, the associated graded
terms are $\mathcal O_{X_0}$-modules:
\[
 \operatorname{gr}_{\mathrm{F}_{X_0}}^k\Omega^s_{D/S}\cong (\ojj^{[k-s]}/\ojj^{[k-s+1]})
                      \otimes_{\OO_{X_0}}\Omega^s_{Z_0/S_0}|_{X_0}.
\]
The induced differential $d_{\operatorname{gr}}$ on
$\operatorname{gr}_{\mathrm{F}_{X_0}}^k\Omega^\bullet_{D/S}$ is described as
follows. For $0\leq s<k$, let $[a]$ be the class of a local section
$a\in\ojj^{[k-s]}$, and let $\widetilde\omega\in\Omega^s_{Z/S}$ be a
local lift of $\omega\in\Omega^s_{Z_0/S_0}|_{X_0}$. The term
$a\otimes d\widetilde\omega$ lies in
$\mathrm{F}_{X_0}^{k+1}\Omega^{s+1}_{D/S}$ and therefore vanishes in the quotient.
Thus
\[
d^s_{\operatorname{gr}}([a]\otimes\omega)
 =[da\wedge\widetilde\omega].
\]


Let $\pi_X:D\rightarrow X$ and $\pi_Y:D\rightarrow Y$ be the natural
projections. On the underlying Zariski space of $X_0$, they induce
filtered homomorphisms
\[
f_0^{-1}(\Omega^\bullet_{Y/S},\mathrm{F}_{Y_0})
 \xrightarrow{\pi_Y^*}( \Omega^\bullet_{D/S},\mathrm{F}_{X_0})
 \xleftarrow{\pi_X^*}(\Omega^\bullet_{X/S},\mathrm{F}_{X_0}).
\]
It follows from the filtered Poincar\'e lemma~\cite[Theorem 7.2]{BO} that 
\[
 \mathrm{R}u_{X_0/S,*}\jj^{[a]}_{X_0/S}\simeq \mathrm{F}_{X_0}^a\Omega^{\bullet}_{D/S}.
\]
Hence \(\pi_X^*\) is a filtered quasi-isomorphism, i.e.,
\[   
s_k:= \operatorname{gr}^k\pi_X^*\colon \operatorname{gr}_{\mathrm{F}_{X_0}}^k\Omega^\bullet_{X/S}\xrightarrow{\sim} \operatorname{gr}_{\mathrm{F}_{X_0}}^k\Omega^\bullet_{D/S}, \quad \forall k\geq 0.
\]
Consider the following diagram
\begin{equation}
\label{eq:rhof0}
\begin{tikzcd}[column sep=2em,row sep=3em]
 f_0^*\operatorname{gr}^k_{\mathrm{F}_{Y_0}}\Omega^\bullet_{Y/S}
   \ar[r,"\operatorname{gr}^k\pi_Y^*"]
 & \operatorname{gr}^k_{\mathrm{F}_{X_0}}\Omega^\bullet_{D/S}
 & \operatorname{gr}^k_{\mathrm{F}_{X_0}}\Omega^\bullet_{X/S}
   \ar[l,"s_k"',"\sim"] \ar[d,"t_k"] \\
 f_0^*\Omega^k_{Y_0/S_0}[-k]
   \ar[u,"i_k=f_0^*\iota_{Y,k}"]
   \ar[rr,dashed,"\sigma_k(f_0)"']
 && (\II\otimes_{\OO_{S_0}}\Omega^{k-1}_{X_0/S_0})[1-k].
\end{tikzcd}
\end{equation}
Here $t_k=\tau_{X,k}$ and $i_k=f_0^*\iota_{Y,k}$ are given
by~\eqref{eq:incl-proj}. Writing 
$y_k=(\operatorname{gr}^k\pi_Y^*)\circ i_k$, the composition
$\sigma_k(f_0)=t_k\,s_k^{-1}\,y_k$ is precisely the map 
$\widetilde\rho_k(f_0)$.

\textbf{Graph conormal extension.}
Let $J_0$ be the ideal of the graph \(X_0\) in $Z_0$, and let $J$ 
be the ideal of the graph \(X_0\) in $Z$. The graph conormal 
sequence is~\cite[(3.21.1)]{derham}
\[
  0\to \II\otimes_{\OO_{S_0}}\mathcal O_{X_0}\to J/J^2\to J_0/J_0^2\to 0.
\]
A splitting corresponds to a flat lifting of the graph in $Z$, which
projects isomorphically to $X$~\cite[proof of Theorem 3.21]{derham}.
Thus the extension represents the obstruction class 
\[
\operatorname{ob}(f_0)\in \Ext^1_{X_0}\bigl(f_0^*\Omega^1_{Y_0/S_0}, \II\otimes_{\OO_{S_0}} \OO_{X_0}\bigr).
\]
Here \(f_0^*\Omega^1_{Y_0/S_0}\cong J_0/J_0^2\). The pushout along 
$\II\to\II/\II^{[2]}$ gives the PD conormal extension
\cite[p.~186]{derham}
\begin{equation}
\label{eq:PD-extension}
 0\to (\II/\II^{[2]})\otimes_{\OO_{S_0}} \mathcal O_{X_0}\rightarrow \ojj/\ojj^{[2]}\xrightarrow{q}f_0^*\Omega^1_{Y_0/S_0}\to0 .
\end{equation}
Under our hypothesis $\II/\II^{[2]}=\II$, this extension is isomorphic
to the previous one. 

Since $f_0^*\Omega^1_{Y_0/S_0}$ is locally free, choose a finite affine 
open cover $\mathfrak U=(U_i)$ of $X_0$ on which~\eqref{eq:PD-extension}
splits. For each $i$, choose a splitting
\[
e_i:f_0^*\Omega^1_{Y_0/S_0}|_{U_i}\rightarrow
(\ojj/\ojj^{[2]})|_{U_i}
\] 
and let
\[
m_i:(\ojj/\ojj^{[2]})|_{U_i}\rightarrow
(\II\otimes_{\OO_{S_0}}\OO_{X_0})|_{U_i}
\] 
be the corresponding retraction, so $1-e_iq$ is $m_i$ followed by the
inclusion of the kernel.
On the overlap $U_{ij}=U_i\cap U_j$, the difference $e_j-e_i$ defines a map 
\[
c_{ij}:=e_j-e_i: f_0^*\Omega^1_{Y_0/S_0}|_{U_{ij}}\to \II\otimes_{\OO_{S_0}}\OO_{X_0}|_{U_{ij}}
\] 
satisfying
\[
m_j-m_i=-c_{ij}q.
\]
With the \v Cech differential $(\delta e)_{ij}=e_j-e_i$, the cocycle
$(c_{ij})$ represents the extension class $\operatorname{ob}(f_0)$.

In the rest of this section, we construct a quasi-inverse 
\[
  R\colon \operatorname{gr}_{\mathrm{F}_{X_0}}^k\Omega^\bullet_{D/S}\to
  \operatorname{gr}_{\mathrm{F}_{X_0}}^k\Omega^\bullet_{X/S}
\]
of \(s_k\) in \(D(\OO_{X_0})\). We use it to compute the cup product
in Proposition~\ref{prop:rho-ob-commute}. 
 Using $\Omega^1_{Z_0/S_0}|_{X_0}=\Omega^1_{X_0/S_0}\oplus
f_0^*\Omega^1_{Y_0/S_0}$, define
\[
\begin{gathered}
 n:f_0^*\Omega^1_{Y_0/S_0}\rightarrow\Omega^1_{Z_0/S_0}|_{X_0},
 \qquad n(v)=(-df_0(v),v),\\
 \mu:\Omega^1_{Z_0/S_0}|_{X_0}\rightarrow\Omega^1_{X_0/S_0},
 \qquad\mu(u,v)=u+df_0(v).
 \end{gathered}
\]
Then \(\mu \circ n=0\). The first graded piece 
\(\operatorname{gr}_{\mathrm{F}_{X_0}}^1\Omega^\bullet_{D/S}\)
is the two-term complex
\[
\ojj/\ojj^{[2]}\xrightarrow{d}\Omega^1_{Z_0/S_0}|_{X_0}.
\]
The differential kills the base ideal $\II$. For a local $S$-lifting
$f_i$ of $f_0$, on a graph equation
$\operatorname{pr}_Y^*y-\operatorname{pr}_X^*f_i^*y$ it induces
$(-df_0(d\bar y),d\bar y)$, so it factors as
\[
 d: \ojj/\ojj^{[2]} \xrightarrow{q} f_0^*\Omega^1_{Y_0/S_0}
 \xrightarrow{n} \Omega^1_{Z_0/S_0}|_{X_0}
\] 

On each \(U_i\), define a degree-zero map
\[
r_i:\operatorname{gr}_{\mathrm{F}_{X_0}}^k\Omega^\bullet_{D/S}|_{U_i}\to \operatorname{gr}_{\mathrm{F}_{X_0}}^k\Omega^\bullet_{X/S}|_{U_i}
\]
by the following components (with the top row omitted for $k=1$).
\[
\begin{tikzcd}[column sep=4em,row sep=2em]
 (\ojj^{[2]}/\ojj^{[3]})\otimes_{\OO_{X_0}}
    \Omega^{k-2}_{Z_0/S_0}|_{X_0}
    \ar[r,"0"] \ar[d,"d"'] & 0 \ar[d,"0"]\\
 (\ojj/\ojj^{[2]})\otimes_{\OO_{X_0}}
    \Omega^{k-1}_{Z_0/S_0}|_{X_0}
    \ar[r,"m_i\otimes\Lambda^{k-1}\mu"] \ar[d,"d"']
 & \II\otimes_{\OO_{S_0}}\Omega^{k-1}_{X_0/S_0} \ar[d,"0"]\\
 \Omega^k_{Z_0/S_0}|_{X_0} \ar[r,"\Lambda^k\mu"]
 & \Omega^k_{X_0/S_0}.
\end{tikzcd}
\]
All components \(r_i^s=0\) for degrees \(s<k-1\). We write
\[
 r_i^k=\Lambda^k\mu,\qquad
 r_i^{k-1}=m_i\otimes\Lambda^{k-1}\mu,\qquad
 r_i^s=0\quad(s\leq k-2).
\]
On the overlap \(U_{ij}\), we define a degree $(-1)$-map 
\[
K_{ij}\colon \operatorname{gr}_{\mathrm{F}_{X_0}}^k\Omega^s_{D/S}|_{U_{ij}}\to \operatorname{gr}_{\mathrm{F}_{X_0}}^k\Omega^{s-1}_{X/S}|_{U_{ij}}, \quad s\le k
\]
whose only possibly nonzero component is in degree $s=k$ and is given by
\[
K_{ij}^k(w_1\wedge\cdots\wedge w_k)
 =\sum_{a=1}^k(-1)^a c_{ij}(\operatorname{pr}_Y w_a)\otimes
 \mu(w_1)\wedge\cdots\widehat{\mu(w_a)}\cdots\wedge\mu(w_k).
\]
Here $\operatorname{pr}_Y:\Omega^1_{Z_0/S_0}|_{X_0}\rightarrow
f_0^*\Omega^1_{Y_0/S_0}$ is the natural projection.

Let \(\operatorname{Tot}\check{\mathcal C}^{\bullet}\!\left(
 \mathfrak U,\operatorname{gr}_{\mathrm{F}_{X_0}}^k\Omega^\bullet_{X/S}\right)\)
denote the total complex of the \v Cech resolution of the complex
\(\operatorname{gr}_{\mathrm{F}_{X_0}}^k\Omega^\bullet_{X/S}\) with respect to 
the open cover \(\mathfrak U=(U_i)\). The local maps $(r_i)$ and $(K_{ij})$ 
give rise to a map 
\begin{equation}
\label{eq:inverse-chian-map}
 R:\operatorname{gr}_{\mathrm{F}_{X_0}}^k\Omega^\bullet_{D/S}
 \rightarrow
 \operatorname{Tot}\check{\mathcal C}^{\bullet}\!\left(
 \mathfrak U,\operatorname{gr}_{\mathrm{F}_{X_0}}^k\Omega^\bullet_{X/S}\right),
\end{equation}
whose component in total degree \(s\) is 
\[
 R^s=\bigl((r_i^s)_i, (K_{ij}^s)_{i<j},0,\ldots\bigr).
\]
The first two components have \v Cech degrees $0$ and $1$ and 
the graded de Rham degrees $s$ and $s-1$.

\begin{lemma}
\label{lem:cech-inverse}
The map \(R\) defined in~\eqref{eq:inverse-chian-map} is a chain map.
If $\varepsilon$ denotes the augmentation to the \v Cech resolution of
\(\operatorname{gr}_{\mathrm{F}_{X_0}}^k\Omega^\bullet_{X/S}\), then the following
diagram commutes.
\[
\begin{tikzcd}[column sep=3em,row sep=3em]
 \operatorname{gr}_{\mathrm{F}_{X_0}}^k\Omega^\bullet_{X/S}
   \ar[r,"s_k","\sim"'] \ar[dr,"\varepsilon"',"\sim"]
 & \operatorname{gr}_{\mathrm{F}_{X_0}}^k\Omega^\bullet_{D/S} \ar[d,"R"]\\
 & \operatorname{Tot}\check{\mathcal C}^{\bullet}\!\left(
   \mathfrak U,\operatorname{gr}_{\mathrm{F}_{X_0}}^k\Omega^\bullet_{X/S}\right).
\end{tikzcd}
\]
\end{lemma}

\begin{proof}
We first prove that $R$ is a chain map. Use the total differential
$\delta+(-1)^i d$ in \v Cech degree $i$. The differential of
$\operatorname{gr}_{\mathrm{F}_{X_0}}^k\Omega^\bullet_{X/S}$ is zero, so the
chain condition reduces to
\[
 r_i^s\circ d=0,\qquad
 K_{ij}^k\circ d=r_j^{k-1}-r_i^{k-1},\qquad
 \delta K=0.
\]
For the first identity, the graded differential satisfies
\[
 d(a\otimes\omega)=nq(a)\wedge\omega,
 \qquad a\in\ojj/\ojj^{[2]},
\]
since differentiating the form factor lands in the next filtered
piece. Thus $r_i^k d=0$ by $\mu n=0$. For $k\geq2$, the quotient
$\ojj^{[2]}/\ojj^{[3]}$ is generated by products $ab$ and divided powers
$\gamma_2(a)$. The identities in the associated graded
\[
 d(ab)=a\,nq(b)+b\,nq(a),\qquad
 d(\gamma_2(a))=a\,nq(a)
\]
show that $r_i^{k-1}d=0$, again by $\mu n=0$. For $k=1$, there is no
incoming differential into degree $0$.

For the second identity, let $a\in\ojj/\ojj^{[2]}$ and let $\omega$
be a local $(k-1)$-form. Since \(da=nq(a)\) and \(\mu n=0\), the formula
for \(K^k_{ij}\) gives that
\[
K_{ij}^k d(a\otimes\omega)=-c_{ij}(q(a))\otimes\Lambda^{k-1}\mu(\omega) =(r_j^{k-1}-r_i^{k-1})(a\otimes\omega)
\]
where the last equality follows from \(m_j-m_i=-c_{ij}q\).
Finally, $\delta c=0$ implies $\delta K=0$, because the formula for
$K_{ij}$ is linear in $c_{ij}$. Hence
\[
 r_j-r_i=K_{ij}d,\qquad \delta K=0.
\]
Thus $R$ commutes with the differentials.

We next check that $R$ is a quasi-inverse of \(s_k\). The description of \(\mu\) and
\(m_i\) implies that \((\Lambda^k\mu)\circ \operatorname{pr}^*_X \) is the identity 
on \(\Omega^k_{X_0/S_0}\) and \(m_i\) restricts to the identity on 
$\II\otimes_{\OO_{S_0}}\OO_{X_0}$. Hence we have $r_i s_k=1$ in both nonzero degrees.
Moreover, $\operatorname{pr}_Y$ kills forms pulled back from $X_0$, 
which implies $K_{ij}s_k=0$. It follows that $R s_k=\varepsilon$ by the relations
\(r_i s_k=1, K_{ij}s_k=0\). Consequently, 
\[
R=\varepsilon \circ s_k^{-1}
\]
can be viewed as an inverse of \(s_k\) in \(D(\OO_{X_0})\).
\end{proof}

\begin{proof}[Proof of Proposition~\ref{prop:rho-ob-commute}]
By~\eqref{eq:rhof0}, $\sigma_k(f_0)$ is the composite
\[
t_k s_k^{-1} y_k. 
\]
Lemma~\ref{lem:cech-inverse} gives the explicit expression of \(s_k^{-1}\).
The \v Cech degree-zero component of $t_k R y_k$ vanishes, and its
degree-one component is $(t_kK_{ij}y_k)_{i<j}$. Let \(v_1\wedge\cdots\wedge v_k\) 
be a local \(k\)-form of \(f_0^*\Omega^k_{Y_0/S_0}\). We have
\[
\begin{aligned}
 (t_kK_{ij}y_k)(v_1\wedge\cdots\wedge v_k)
 ={}&\sum_{a=1}^k(-1)^a c_{ij}(v_a)\otimes\\
 &df_0(v_1)\wedge\cdots\widehat{df_0(v_a)}\cdots\wedge df_0(v_k).
\end{aligned}
\]
Put $q=m-k$ and use the total complex differential $\delta+(-1)^q d$, with the
\v Cech differential $(\delta a)_{i_0\ldots i_{q+1}}=\sum_{j=0}^{q+1}(-1)^j
a_{i_0\ldots\widehat{i_j}\ldots i_{q+1}}$.
For a \v Cech $q$-cocycle $(a_{i_0\ldots i_q})$ representing the pullback 
of a class in $\Ho^q(Y_0,\Omega^k_{Y_0/S_0})$ on the cover $\mathfrak U$ of
$X_0$, the degree-one component of $R$ gives
\[
 (-1)^q(t_kK_{i_0i_1}y_k)
       (a_{i_1\ldots i_{q+1}}).
\]
Applying $c_{ij}$ to the first factor of the exterior
map~\eqref{eq:inclusion} and $df_0$ to the remaining factors gives
exactly $t_kK_{ij}y_k$, with coefficients $(-1)^a$. Since $(c_{ij})$
represents $\operatorname{ob}(f_0)$, the totalization gives
\[
\rho(f_0)_k=(-1)^{m-k}\,\operatorname{ob}(f_0)\cup_{f_0}-.
\]
The derived comparison identifies this \v Cech calculation with
$\Ho^m(\mathrm{R}\Gamma(X_0,\sigma_k(f_0)))$ after pullback along $f_0$.
\end{proof}

\section{\texorpdfstring{$p$}{p}-adic lifting of automorphisms}
\label{bloch}

\begin{lemma}\label{lem:witt-pd-vanishing}
Let $k$ be a perfect field of characteristic $p>2$, and $W=W(k)$
the Witt vector of \(k\). Write $W_n=W/p^{n+1}W$. Let
\[
I_n=p^{n+1}W/p^{n+2}W=\Ker(W_{n+1}\rightarrow W_n).
\]
carry the canonical divided power structure induced from $pW$.
Then we have $I_n^{[a]}=0$ for every $a\geq2$.
\end{lemma}

\begin{proof}
The canonical divided powers \(\gamma_j\) on \(pW\subset W\) descend
to \(I_n\). Suppose that
\(p^{n+1}z\in p^{n+1}W\) represents \(x\in I_n\). Then
\[
 \gamma_j(p^{n+1}z)=\frac{p^{j(n+1)}z^{j}}{j!}.
\]
Since $v_p(j!) < j/(p-1)$, we obtain the inequality
\[
v_p(p^{j(n+1)})-v_p(j!) > j(n+1)-j/(p-1).
\]
Since $p>2$, for \(j\ge 2\) and \(n\ge 1\) we have
\[
 j(n+1)-\frac{j}{p-1}\ge j\left(n+\frac12\right)
 \ge 2n+1\ge n+2.
\]
For \(j\ge 2\) and  \(n=0\), the inequality \(j-j/(p-1)\ge 1\) holds
for every odd prime \(p\), so \(v_p(p^{j})-v_p(j!)\ge 2\).
In any case \(\gamma_j(p^{n+1}z)\in p^{n+2}W\). 

The divided power ideal has the form 
\[
I_n^{[2]}
=I_n^2+\bigl(\gamma_j(x)\mid x\in I_n,\ j\geq2\bigr)
\]
where the parentheses denote the ideal generated in $W_{n+1}$
\cite[Section 60.6]{StacksPD}. Since $I_n^2=0$, it follows that $I_n^{[a]}=0$ for $a\ge2$.

\end{proof}

\begin{proof}[Proof of Theorem~\ref{corolifting}]
Let $p=\operatorname{char}(k)>2$, let $W=W(k)$, and let $W_n$ and
$I_n$ be as in Lemma~\ref{lem:witt-pd-vanishing}. Write
$X_n=X\times_W\operatorname{Spec}W_n$ for the $n$-th reduction of $X/W$.
By Lemma~\ref{lem:witt-pd-vanishing}, the canonical divided powers on $I_n$
satisfy $I_n^{[a]}=0$ for $a\geq2$.

Suppose inductively that $g_n$ is a $W_n$-automorphism of $X_n$ lifting $g_0$.
Its obstruction to lifting to $X_{n+1}$ lies in
\[
 \Ho^1\bigl(X_n,g_n^*T_{X_n/W_n}\otimes_{W_n} I_n\bigr)
 \cong \Ho^1(X_0,g_0^*T_{X_0/k})\otimes_k I_n.
\]
The isomorphism follows from $pI_n=0$ and the restriction of
$g_n^*T_{X_n/W_n}$ to $X_0$. Since $g_0^*$ is an isomorphism
on Hodge cohomology, tensoring the infinitesimal
Torelli injection~\eqref{cupproduct} with $I_n$ over $k$ gives an injection
\[
 \widehat\Psi_n:\;
 \Ho^1(X_0,g_0^*T_{X_0/k})\otimes_k I_n
 \rightarrow \bigoplus_{r+q=m}
 \operatorname{Hom}\bigl(\Ho^q(X_0,\Omega^r_{X_0/k}),\Ho^{q+1}(X_0,\Omega^{r-1}_{X_0/k})\otimes_k I_n\bigr).
\]
Apply Proposition~\ref{prop:rho-ob-commute} with $S=\Spec W_{n+1}$, 
$S_0=\Spec W_n$, \(\II=I_n\) and $f_0=g_n$. The hypotheses on de Rham 
and Hodge cohomology hold after base change from $W$. The map defined
in~\eqref{eq:Hodge-PD-defect} is
\[
 \rho(g_n)_r:
 \Ho^q(X_n,\Omega^r_{X_n/W_n})
 \rightarrow
 \Ho^{q+1}(X_n,\Omega^{r-1}_{X_n/W_n})\otimes_{W_n} I_n.
\]
The Hodge cohomology base change isomorphisms and $pI_n=0$ give
the commutative diagram
\[
\begin{tikzcd}[column sep=large,row sep=large]
 \Ho^q(X_n,\Omega^r_{X_n/W_n})
   \ar[r,"\rho(g_n)_r"] \ar[d,two heads,"\mathrm{res}"']
 & \Ho^{q+1}(X_n,\Omega^{r-1}_{X_n/W_n})\otimes_{W_n}I_n
   \ar[d,"\sim"]\\
 \Ho^q(X_0,\Omega^r_{X_0/k})
   \ar[r,"{(-1)^q\operatorname{ob}(g_n)\cup_{g_0} -}"]
 & \Ho^{q+1}(X_0,\Omega^{r-1}_{X_0/k})\otimes_k I_n .
\end{tikzcd}
\]

The morphism \(\cris^m(g_0)\) preserves the Hodge filtration on 
\(\cris^m(X_0/W)\). The quotient map of PD rings 
\((W,(p))\rightarrow (W_{n+1},(p))\) with the canonical divided powers 
induces a base change isomorphism
\[
 \alpha_n:\cris^m(X_0/W)\otimes_W W_{n+1}
 \xrightarrow{\sim}
 \Ho^m_{\mathrm{cris}}(X_0/W_{n+1})
\]
by crystalline--de Rham comparison for the smooth lift $X/W$ and the
base change. In particular, 
\[
\alpha_n(\F^r \Ho^m_{\DR}(X/W)\otimes_W W_{n+1}) = \F^r\Ho^m_{\DR}(X_{n+1}/W_{n+1})
\]
By functoriality, the action $\cris^m(g_0)$ preserves 
\(\F^r\Ho^m_{\DR}(X_{n+1}/W_{n+1})\).
The compatible PD morphism $(W_{n+1},I_n)\to(W_{n+1},(p))$ induces a comparison
\[
 \beta_n:
 \Ho^m_{\mathrm{cris}}(X_n/W_{n+1})
 \xrightarrow{\ \sim\ }
 \Ho^m_{\mathrm{cris}}(X_0/W_{n+1}).
\]
Both sides are computed by the de Rham complex of the common smooth
lift $X_{n+1}/W_{n+1}$; see~\cite[Theorem 7.2]{BO}. In particular,
$\beta_n$ identifies the transported Hodge filtrations 
\(\F^r\Ho^m_{\DR}(X_{n+1}/W_{n+1})\).
Since $g_n$ reduces to $g_0$, we have
\[
 \beta_n\circ \cris^m(g_n)
 =\cris^m(g_0)\circ\beta_n.
\]
Thus $\cris^m(g_n)$ preserves the Hodge filtration on 
\(\Ho^m_{\mathrm{cris}}(X_n/W_{n+1})\). By the definition
of the filtration defect in~\eqref{defrohf},
$\rho(g_n)_r=0$ for every $r$. Therefore, the map of cup product
\(\operatorname{ob}(g_n)\cup_{g_0}-\) vanishes. Injectivity
of \(\widehat\Psi_n\) implies $\operatorname{ob}(g_n)=0$,
so $g_n$ lifts to an automorphism $g_{n+1}$. Induction gives
a formal automorphism $\widehat g$ of the $p$-adic 
completion $\widehat X$. Since $X$ is proper over the complete 
ring $W$, Grothendieck's 
algebraization theorem~\cite[Lemma~30.28.3]{StacksGE}
algebraizes $\widehat g$ to an automorphism \(g\) of $X/W$.
\end{proof}

\begin{cor} \label{padicfaithful}
Keep the notation and the hypotheses on $k$, $X$, and $m$ from
Theorem~\ref{corolifting}.
Let $K=\operatorname{Frac}W(k)$, let $\ell\neq p$ be prime, and write $X_{\bar k}$ and $X_{\bar K}$
for the geometric special and generic fibers, respectively.
If $g_0$ is a $k$-automorphism of $X_0$ of finite order and acts trivially
on $\et^m(X_{\bar k},\mathbb Q_\ell)$, then $g_0$ lifts to a
$W(k)$-automorphism of $X$.
In particular, if $\Aut_k(X_0)$ is finite and $\Aut_K(X_K)$ acts
faithfully on $\et^m(X_{\bar K},\mathbb Q_\ell)$, then $\Aut_k(X_0)$
acts faithfully on $\et^m(X_{\bar k},\mathbb Q_\ell)$.
\end{cor}

\begin{proof}
The characteristic polynomials of $g_0^*$ on crystalline and
$\ell$-adic cohomology are obtained from the same polynomial in
$\mathbb Z[t]$ \cite[3.7.3 and 3.10]{Ill1}, \cite[Proposition~2.1(1)]{Keum}. Thus
\[
 \det(\Id-tg_0^*;\cris^m(X_0/W)_K)
 =\det(\Id-tg_0^*;\et^m(X_{\bar k},\mathbb Q_\ell)).
\]
Here we apply the comparison over $\bar k$ and use crystalline base change
along $W(k)\to W(\bar k)$ to identify the polynomial on the left.
Since $g_0$ has finite order, its action on each characteristic-zero
cohomology space is semisimple. Hence the assumed identity on
$\ell$-adic cohomology implies $\cris^m(g_0)_K=\Id$.
The $E_1$-degeneration and freeness hypotheses make $\cris^m(X_0/W)$
a free $W$-module, so the integral action is also the identity. In
particular it preserves the Hodge filtration, and
Theorem~\ref{corolifting} gives the lift.

For the last assertion, smooth proper base change identifies the
$\ell$-adic actions of $g_0$ and $g_K$ for a compatible choice of geometric
points. Faithfulness on the generic fiber gives $g_K=\Id$. Flatness over $W$
makes the generic fiber schematically dense in $X$, and separatedness
then gives $g=\Id$ and
$g_0=\Id$.
\end{proof}



\end{document}